\documentclass[reqno,a4paper,11pt]{amsart}
\usepackage[all,poly]{xy}
\usepackage{amsfonts}
\usepackage{amssymb}
\usepackage{amsmath}
\usepackage{color}
\usepackage[pagebackref,colorlinks]{hyperref}
\usepackage{enumerate}
     
\theoremstyle{plain}
\newtheorem {lemma}{Lemma}

\newtheorem {theorem}[lemma]{Theorem}
\newtheorem {corollary}[lemma]{Corollary}

\newtheorem {replacement lemma}[lemma]{Replacement Lemma}

\theoremstyle{definition}
\newtheorem{definition}[lemma]{Definition}

\newtheorem{remark}[lemma]{Remark}
\newtheorem {example}[lemma]{Example}

\newcommand{\N}{\mathbb{N}}
\newcommand{\X}{\langle X \rangle}
\newcommand{\GKdim}{\operatorname{GKdim}}
\newcommand{\st}{\operatorname{st}}
\newcommand{\G}{\operatorname{\mathcal{G}^w}}
\newcommand{\M}{\operatorname{\mathcal{M}^{ab}}}
\newcommand{\A}{\operatorname{\mathcal{A}_K}}
\newcommand{\Mat}{\operatorname{\mathbb{M}}}

\newcommand{\im}{\operatorname{im}}
\newcommand{\id}{\operatorname{id}}

\title{The $V$-monoid of a weighted Leavitt path algebra}

\author{Raimund Preusser}
\address{Department of Mathematics,
University of Brasilia, Brazil}
\email{raimund.preusser@gmx.de}

\subjclass[2000]{16S10, 16W10, 16W50, 16D70} 
\keywords{Weighted Leavitt path algebra, finitely generated projective modules, $K_0$}

\begin{document}

\begin{abstract} 
We compute the $V$-monoid of a weighted Leavitt path algebra of a row-finite weighted graph, correcting a wrong computation of the $V$-monoid that exists in the literature. Further we show that the description of $K_0$ of a weighted Leavitt path algebra that exists in the literature is correct (although the computation was based on a wrong $V$-monoid description).
\end{abstract}

\maketitle

\section{Introduction}
The weighted Leavitt path algebras (wLpas) were introduced by R. Hazrat in \cite{hazrat13}. They generalise the Leavitt path algebras (Lpas). While the Lpas only embrace Leavitt's algebras $L_K(1,1+k)$ where $K$ is a field and $k\geq 0$, the wLpas embrace all of Leavitt's algebras $L_K(n,n+k)$ where $K$ is a field, $n\geq 1$ and $k\geq 0$. In \cite{hazrat-preusser} linear bases for wLpas were obtained. They were used to classify the simple and graded simple wLpas and the wLpas which are domains. In \cite{preusser} the Gelfand-Kirillov dimension of a weighted Leavitt path algebra $L_K(E,w)$, where $K$ is a field and $(E,w)$ is a row-finite weighted graph, was determined. Further finite-dimensional wLpas were investigated. In \cite{preusser1} locally finite wLpas were investigated.

The $V$-monoid $V(R)$ of an associative, unital ring $R$ is the set of all isomorphism classes of finitely generated projective right $R$-modules, which becomes an abelian monoid by defining $[P]+[Q]:=[P\oplus Q]$ for any $[P],[Q]\in V(R)$. It can also be defined using matrices and the definition can be extended to include all associative rings, see Section 4. For an associative ring $R$ with local units, the Grothendieck group $K_0(R)$ is the group completion $V(R)^+$ of $V(R)$.

A presentation for $V(L_K(E,w))$ where $K$ is a field and $(E,w)$ is a row-finite weighted graph was given in \cite[Theorem 5.21]{hazrat13}. In \cite{hazrat-preusser} this presentation was used in order to show that there is a huge class $C$ of wLpas that are domains but neither are isomorphic to an Lpa nor to a Leavitt algebra. Unfortunately, \cite[Theorem 5.21]{hazrat13} is wrong as we will show in Section 4. In this paper we correct the false \cite[Theorem 5.21]{hazrat13}. It turns out that the statement of \cite[Theorem 5.21]{hazrat13} is true at least for row-finite weighted graphs $(E,w)$ that have the property that for any vertex $v\in E^0$ all the edges in $s^{-1}(v)$ have the same weight. Further it turns out, surprisingly, that \cite[Theorem 5.23]{hazrat13}, which gives a presentation for the Grothendieck group $K_0(L_K(E,w))$ of a wLpa, is correct. 

The rest of this paper is organised as follows. In Section 2 we recall some standard notation which is used throughout the paper. In Section 3 we recall the definition of a wLpa. In Section 4 we prove our main result Theorem \ref{thmm} and show that the description of $K_0(L_K(E,w))$ obtained in \cite{hazrat13} is correct. Moreover, we show that there is a class $D$, containing the class $C$ mentioned above, that consists of wLpas that are domains but neither are isomorphic to an Lpa nor to a Leavitt algebra. In the last section we determine the $V$-monoids of some concrete examples of wLpas.
\section{Notation}
Throughout the paper $K$ denotes a field. $\N$ denotes the set of positive integers and $\N_0$ the set of nonnegative integers. If $m,n\in\N$ and $R$ is a ring, then $\Mat_{m\times n}(R)$ denotes the set of $m\times n$-matrices whose entries are elements of $R$. Instead of $\Mat_{n\times n}(R)$ we might write $\Mat_{n}(R)$. 
\section{Weighted Leavitt path algebras}

\begin{definition}\label{defdg}
A {\it directed graph} is a quadruple $E=(E^0,E^1,s,r)$ where $E^0$ and $E^1$ are sets and $s,r:E^1\rightarrow E^0$ maps. The elements of $E^0$ are called {\it vertices} and the elements of $E^1$ {\it edges}. If $e$ is an edge, then $s(e)$ is called its {\it source} and $r(e)$ its {\it range}. $E$ is called {\it row-finite} if $s^{-1}(v)$ is a finite set for any vertex $v$ and {\it finite} if $E^0$ and $E^1$ are finite sets.
\end{definition}
\begin{definition}
A {\it weighted graph} is a pair $(E,w)$ where $E$ is a directed graph and $w:E^1\rightarrow \N$ is a map. If $e\in E^1$, then $w(e)$ is called the {\it weight} of $e$. $(E,w)$ is called {\it row-finite} (resp. {\it finite}) if $E$ is row-finite (resp. finite). In this article all weighted graphs are assumed to be row-finite. For a vertex $v\in E^0$ we set $w(v):=\max\{w(e)\mid e\in s^{-1}(v)\}$ with the convention $\max \emptyset=0$.
\end{definition}
\begin{remark}
In \cite{hazrat13} and \cite{hazrat-preusser}, $E^1$ was denoted by $E^{\st}$. The set $\{e_i \mid e\in E^1, 1\leq i\leq w(e)\}$ was denoted by $E^1$.
\end{remark}
\begin{definition}\label{def3}
Let $(E,w)$ be a weighted graph. The associative $K$-algebra presented by the generating set $E^0\cup \{e_i,e_i^*\mid e\in E^1, 1\leq i\leq w(e)\}$ and the relations
\begin{enumerate}[(i)]
\item $uv=\delta_{uv}u\quad(u,v\in E^0)$,
\item $s(e)e_i=e_i=e_ir(e),~r(e)e_i^*=e_i^*=e_i^*s(e)\quad(e\in E^1, 1\leq i\leq w(e))$,
\item $\sum\limits_{e\in s^{-1}(v)}e_ie_j^*= \delta_{ij}v\quad(v\in E^0,1\leq i, j\leq w(v))$ and
\item $\sum\limits_{1\leq i\leq w(v)}e_i^*f_i= \delta_{ef}r(e)\quad(v\in E^0,e,f\in s^{-1}(v))$
\end{enumerate}
is called {\it weighted Leavitt path algebra (wLpa) of $(E,w)$} and is denoted by $L_K(E,w)$. In relations (iii) and (iv), we set $e_i$ and $e_i^*$ zero whenever $i > w(e)$. 
\end{definition}


\begin{example}\label{exex1}
If $(E,w)$ is a weighted graph that $w(e)=1$ for all $e \in E^{1}$, then $L_K(E,w)$ is isomorphic to the usual Leavitt path algebra $L_K(E)$. 
\end{example}

\begin{example}\label{wlpapp}
Let $n\geq 1$ and $k\geq 0$. If $(E,w)$ is a weighted graph with precisely one vertex and precisely $n+k$ edges each of which has weight $n$, then $L_K(E,w)$ is isomorphic to the Leavitt algebra $L_K(n,n+k)$, for details see \cite[Example 4]{hazrat-preusser}. 
\end{example}

\begin{remark}
Let $(E,w)$ be a weighted graph. Then there is an involution $*$ on $L_K(E,w)$ mapping $x\mapsto x$, $v\mapsto v$, $e_i\mapsto e_i^*$ and $e_i^*\mapsto e_i$ for any $x\in K$, $v\in E^0$, $e\in E^1$ and $1\leq i\leq w(e)$, see \cite[Proof of Proposition 5.7]{hazrat13}. If $m,n\in\N$, then $*$ induces a map $\Mat_{m,n}(L_K(E,w))\rightarrow \Mat_{n,m}(L_K(E,w))$ mapping a matrix $\sigma$ to the matrix $\sigma^*$  one gets by transposing $\sigma$ and then applying the involution $*$ to each entry.
\end{remark}
\section{The $V$-monoid of a weighted Leavitt path algebra}
Consider the weighted graphs\\
\[
(E,w):\xymatrix@C+15pt{ u& v\ar[l]_{e,2}\ar[r]^{f,2}& x}\quad\text{ and }\quad(E,w'):\xymatrix@C+15pt{ u& v\ar[l]_{e,1}\ar[r]^{f,2}& x}.
\]\\
Set $L:=L_K(E,w)$ and $L':=L_K(E,w')$. According to \cite[Theorem 5.21]{hazrat13} it is true that
\[V(L)\cong V(L')\cong \N_0^{E^0}/\langle 2\alpha_v=\alpha_u+\alpha_x\rangle\]
where for any $y\in E^0$, $\alpha_y$ denotes the element of $\N_0^{E^0}$ whose $y$-component is one and whose other components are zero. Hence $V(L')$ is not a refinement monoid. But in \cite[Example 47]{preusser} it was shown that $L'\cong M_3(K)\oplus M_3(K)$. Hence $V(L')\cong \N_0^2$ and therefore $V(L')$ is a refinement monoid. In view of this contradiction, is \cite[Theorem 5.21]{hazrat13} wrong?

First we consider the algebra $L$. By the relations for the generators of a wLpa (see Definition \ref{def3}), the matrix $A:=\begin{pmatrix}e_1&f_1\\e_2&f_2\end{pmatrix}\in \Mat_2(L)$ defines a ``universal" (in the sense of ``as general as possible") isomorphism $uL\oplus xL\rightarrow vL\oplus vL$ (by left multiplication). Set 
\begin{align*}
B_0&:=K^{E^0}\text{ and }\\
B_1&:=B_0\langle i,i^{-1}:\overline  {\alpha_uB_0\oplus\alpha_xB_0}\cong \overline  {\alpha_vB_0\oplus\alpha_vB_0}\rangle
\end{align*}
(see \cite[p. 38]{bergman74}) where for any $y\in E^0$, $\alpha_y$ denotes the element of $B_0$ whose $y$-component is one and whose other components are zero. One checks easily that $L\cong B_1$. It follows from \cite[Theorem 5.2]{bergman74} that 
\[V(L)\cong \N_0^{E^0}/\langle \alpha_u+\alpha_x=2\alpha_v\rangle\]
and hence \cite[Theorem 5.21]{hazrat13} yields a correct presentation for $V(L)$. 


For the algebra $L'$ the situation is a bit different. By the relations for the generators of a wLpa, the matrix $A:=\begin{pmatrix}e_1&f_1\\0&f_2\end{pmatrix}\in \Mat_2(L')$ defines an isomorphism $uL'\oplus xL'\rightarrow vL'\oplus vL'$, but this isomorphism is not universal since an entry of $A$ is zero. Since $A$ defines an isomorphism $uL'\oplus xL'\rightarrow vL'\oplus vL'$, we have $vL'\oplus vL'=P\oplus Q$ where $P=\im \begin{pmatrix}e_1\\0\end{pmatrix}$ and $Q=\im\begin{pmatrix}f_1\\f_2\end{pmatrix}$. But this splitting into two direct summands is not universal, since $P$ is already a direct summand of $vL'$. Instead $vL'$ universally breaks up into two direct summands, namely $vL'\cong P\oplus O$ where $O=\im \begin{pmatrix}f_1\\0\end{pmatrix}$. Clearly $e_1$ defines a universal isomorphism $uL'\cong P$ and $\begin{pmatrix}f_1\\f_2\end{pmatrix}$ defines a universal isomorphism $xL'\cong O \oplus vL'$. Hence we can describe $L'$ as follows. Set 
\begin{align*}
B_0&:=K^{E^0},\\
B_1&:=B_0\langle\epsilon:\overline{\alpha_vB_0}\rightarrow\overline{\alpha_vB_0};\epsilon^2=\epsilon\rangle,\\
B_2&:=\langle i,i^{-1}:\overline  {\alpha_u B_1}\cong \overline  {\ker \epsilon}\rangle\text{ and }\\
B_3&:=\langle j,j^{-1}:\overline  {\alpha_x B_2}\cong \overline  {\im \epsilon\oplus \alpha_vB_2}\rangle 
\end{align*} 
(see \cite[pp. 38-39]{bergman74}). One checks easily that $L'\cong B_3$ (for details see Theorem \ref{thmm}, Part II). It follows from \cite[Theorems 5.1, 5.2]{bergman74} that
\[V(L')\cong \N_0^{E^0\cup\{p,q\}}/\langle \alpha_p+\alpha_q=\alpha_v,\alpha_u=\alpha_p,\alpha_x=\alpha_q+\alpha_v\rangle\cong \N_0^2.\]
Thus \cite[Theorem 5.21]{hazrat13} indeed is wrong.

In this section we repair \cite[Theorem 5.21]{hazrat13}. Further we show that \cite[Theorem 5.23]{hazrat13}, which gives a presentation for $K_0(L_K(E,w))$ where $(E,w)$ is any weighted graph, is correct. In particular $K_0(L)\cong K_0(L')$ where $L$ and $L'$ are the wLpas defined above, while $V(L)\not\cong V(L')$.

We denote by $\G$ the category whose objects are all weighted graphs and whose morphisms are the complete weighted graph homomorphisms between weighted graphs (see \cite[p. 884]{hazrat13}). Further we denote the category of associative $K$-algebras by $\A$ and the category of abelian monoids by $\M$. We start by defining three functors, $L_K:\G\rightarrow \A$, $V:\A\rightarrow \M$ and $M:\G\rightarrow \M$. We will then show that $V\circ L_K\cong M$.
\begin{definition}
In Definition \ref{def3} we associated to any weighted graph $(E,w)$ an associative $K$-algebra $L_K(E,w)$. If $\phi:(E,w)\rightarrow (E',w')$ is a morphism in $\mathcal{G}^w$, then there is a unique $K$-algebra homomorphism $L_K(\phi):L_K(E,w)\rightarrow L_K(E',w')$ such that $L_K(\phi)(v)=\phi^0(v)$, $L_K(\phi)(e_i)=(\phi^1(e))_i$ and $L_K(\phi)(e_i^*)=(\phi^1(e))_i^*$ for any $v\in E^0$, $e\in E^1$ and $1\leq i \leq w(e)$. One checks easily that $L_K:\G \rightarrow \A$ is a functor that commutes with direct limits.
\end{definition}
\begin{definition}
Let $A$ be an associative $K$-algebra. Let $\Mat_\infty(A)$ be the directed union of the rings $M_n(A)~(n\in\N)$, where the transition maps $M_n(A)\rightarrow M_{n+1}(A)$ are given by $x\mapsto \begin{pmatrix}x&0\\0&0\end{pmatrix}$. Let $I(\Mat_\infty(A))$ denote the set of all idempotent elements of $\Mat_\infty(A)$. If $e, f\in I(\Mat_\infty(A))$, write $e\sim f$ iff there are $x,y\in \Mat_\infty(A)$ such that
$e = xy$ and $f = yx$. Then $\sim$ is an equivalence relation on $I(\Mat_\infty(A))$. Let $V(A)$ be the set of all $\sim$-equivalence classes, which becomes an abelian monoid by defining
\[[e]+[f]=\left[\begin{pmatrix}e&0\\0&f\end{pmatrix}\right]\]
for any $[e],[f]\in V(A)$. If $\phi:A\rightarrow B$ is a morphism in $\A$, let $V(\phi):V(A)\rightarrow V(B)$ be the canonical monoid homomorphism induced by $\phi$. One checks easily that $V:\A\rightarrow \M$ is a functor that commutes with direct limits.
\end{definition}
\begin{remark}\label{rempro}
Let $A$ be an associative, unital $K$-algebra. Let $V'(A)$ denote the set of isomorphism classes of finitely generated projective right $A$-modules, which becomes an abelian monoid by defining $[P]+[Q]:=[P\oplus Q]$ for any $[P],[Q]\in V'(A)$. Then $V'(A)\cong V(A)$ as abelian monoids, see \cite[Definition 3.2.1]{abrams-ara-molina}.
\end{remark}
\begin{definition}\label{defM}
Let $(E,w)$ be a weighted graph. For any $v\in E^0$ write $w(s^{-1}(v))=\{w_1(v),\dots,$ $w_{k_v}(v)\}$ where $k_v\geq 0$ and $w_1(v)<\dots<w_{k_v}(v)$ (hence $k_v$ is the number of different weights of the edges in $s^{-1}(v)$). Further set $w_0(v):=0$ for any $v\in E^0$ (note that with this convention one has $w_{k_v}(v)=w(v)$ for any $v\in E^0$). Let $M(E,w)$ be the abelian monoid presented by the generating set $\{v,q_1^v,\dots,q^v_{k_v-1}\mid v\in E^0\}$ and the relations
\begin{equation}
q^v_{i-1}+(w_i(v)-w_{i-1}(v))v=q_i^v+\sum\limits_{\substack{e\in s^{-1}(v),\\w(e)=w_i(v)}}r(e)\quad\quad(v\in E^0,1\leq i\leq k_v)
\end{equation}
where $q^v_0=q^v_{k_v}=0$. If $\phi:(E,w)\rightarrow (E',w')$ is a morphism in $\mathcal{G}^w$, then there is a unique monoid homomorphism $M(\phi):M(E,w)\rightarrow M(E',w')$ such that $M(\phi)([v])=[\phi^0(v)]$ and $M(\phi)([q_i^v])=[q_i^{\phi^0(v)}]$ for any $v\in E^0$ and $1\leq i \leq k_v-1$. One checks easily that $M:\G\rightarrow \M$ is a functor that commutes with direct limits.
\end{definition}
\begin{remark}
If $k_v\leq 1$ for any $v\in E^0$, then $M(E,w)$ is the abelian monoid $M_E$ defined in \cite[Theorem 5.21]{hazrat13}.
\end{remark}
\begin{lemma}\label{lemmon}
Let $G$ be an abelian group (resp. an abelian monoid) presented by a generating set $X$ and relations 
\[l_i=r_i~ (i\in I)\text{ and }y=\sum\limits_{x\in X\setminus\{y\}}n_xx\]
where for any $i\in I$, $l_i$ and $r_i$ are elements of the free abelian group (resp. the free abelian monoid) $G\X$ generated by $X$, $y$ is an element of $X$, the $n_x$ are integers (resp. nonnegative integers) and only finitely of them are nonzero. Let $G\langle X\setminus\{y\}\rangle$ be the free abelian group (resp. the free abelian monoid) generated by $X\setminus \{y\}$ and $f:G\X\rightarrow G\langle X\setminus\{y\}\rangle$ the homomorphism which maps each $x\in X\setminus\{y\}$ to $x$ and $y$ to $\sum\limits_{x\in X\setminus\{y\}}n_xx$. Then $G$ is also presented by the generating set $X\setminus\{y\}$ and the relations $f(l_i)=f(r_i)~(i\in I)$.
\end{lemma}
\begin{proof}
Straightforward.
\end{proof}
\begin{theorem}\label{thmm}
$V\circ L_K\cong M$. Moreover, if $(E,w)$ is finite, then $L_K(E,w)$ is left and right hereditary.
\end{theorem}
\begin{proof}
We have divided the proof in two parts, Part I and Part II. In Part I we define a natural transformation $\theta:M\rightarrow V\circ L_K$. In Part II we show that $\theta$ is a natural isomorphism and further that $L_K(E,w)$ is left and right hereditary provided that $(E,w)$ is finite.\\
\\  
{\bf Part I}
Let $(E,w)$ be a weighted graph. Let $v\in E^0$ be a vertex that emits edges (i.e. $s^{-1}(v)\neq \emptyset$). Write $s^{-1}(v)=\{e^{1,v}, \dots, e^{n(v),v}\}$ where $w(e^{1,v})\leq\dots \leq w(e^{n(v),v})$. Let $A=A(v)\in\Mat_{w(v)\times n(v)}(L_K(E,w))$ be the matrix whose entry at position $(i,j)$ is $e^{j,v}_i$ (we set $e^{j,v}_i:=0$ if $i>w(e^{j,v})$). By relations (iii) and (iv) in Definition \ref{def3} we have that
\begin{equation}
AA^*=\begin{pmatrix} v&& \\&\ddots &\\&&v\end{pmatrix}\in\Mat_{w(v)}(L_K(E,w))
\end{equation}
and
\begin{equation}
A^*A=\begin{pmatrix} r(e^{1,v})&& \\&\ddots &\\&&r(e^{n(v),v})\end{pmatrix}\in\Mat_{n(v)}(L_K(E,w)).
\end{equation}
As in Definition \ref{defM}, set $w_0(v):=0$ and write $w(s^{-1}(v))=\{w_1(v),\dots,w_{k_v}(v)\}$ where $w_1(v)<\dots<w_{k_v}(v)$. For any $0\leq l\leq k_v$ set $n_l(v):=|s^{-1}(v)\cap w^{-1}(\{w_0(v), \dots , w_l(v)\})|$ (note that $n_0(v)=0$ and $n_{k_v}(v)=n(v)$). For $0\leq l<t\leq k_v$ and $0\leq l'<t'\leq k_v$ let $A^{n_{l'},n_{t'}}_{w_{l},w_{t}}=A^{n_{l'},n_{t'}}_{w_{l},w_{t}}(v)\in\Mat_{(w_{t}(v)-w_{l}(v))\times (n_{t'}(v)-n_{l'}(v))}(L_K(E,w))$ be the matrix whose entry at position $(i,j)$ is $e^{n_{l'}(v)+j,v}_{w_{l}(v)+i}$. Then $A$ has the block form
\[A=\begin{pmatrix}
A^{n_0,n_1}_{w_0,w_1}&A^{n_1,n_2}_{w_0,w_1}&\dots&A^{n_{k_v-1},n_{k_v}}_{w_0,w_1}\\
0&A^{n_1,n_2}_{w_1,w_2}&\dots&A^{n_{k_v-1},n_{k_v}}_{w_1,w_2}\\
0&0&\ddots&\vdots\\
0&0&0&A^{n_{k_v-1},n_{k_v}}_{w_{k_v-1},w_{k_v}}
\end{pmatrix}\]
and $A^*$ has the block form 
\[A^*=\begin{pmatrix}
(A^{n_0,n_1}_{w_0,w_1})^*&0&0&0\\
(A^{n_1,n_2}_{w_0,w_1})^*&(A^{n_1,n_2}_{w_1,w_2})^*&0&0\\
\vdots&\vdots&\ddots&0\\
(A^{n_{k_v-1},n_{k_v}}_{w_0,w_1})^*&(A^{n_{k_v-1},n_{k_v}}_{w_1,w_2})^*&\hdots&(A^{n_{k_v-1},n_{k_v}}_{w_{k_v-1},w_{k_v}})^*
\end{pmatrix}.\]
For any $1\leq l \leq k_v-1$ set $\epsilon_l=\epsilon_l(v):=A^{n_{l},n_{k_v}}_{w_0,w_l}(A^{n_{l},n_{k_v}}_{w_0,w_l})^*\in \Mat_{w_l(v)}(L_K(E,w))$. 
It follows from equation (2) that 
\begin{equation}
\epsilon_l=\begin{pmatrix} v&& \\&\ddots &\\&&v\end{pmatrix}-A^{n_{0},n_l}_{w_0,w_l}(A^{n_{0},n_l}_{w_0,w_l})^*.
\end{equation}
By equation (3) we have
\begin{equation}
(A^{n_{0},n_l}_{w_0,w_l})^*A^{n_{0},n_l}_{w_0,w_l}=\begin{pmatrix} r(e^{1,v})&& \\&\ddots &\\&&r(e^{n_l(v),v})\end{pmatrix}.
\end{equation}
Equations (4) and (5) imply that $\epsilon_l$ is an idempotent matrix for any $1\leq l\leq k_v-1$.\\
Let $F$ be the free abelian monoid generated by the set $\{v,q_1^v,\dots,q^v_{k_v-1}\mid v\in E^0\}$. There is a unique monoid homomorphism $\psi:F\rightarrow V(L_K(E,w))$ such that $\psi(v)=[(v)]$ and $\psi(q_l^v)=[\epsilon_l(v)]$ for any $v\in E_0$ and $1\leq l \leq k_v-1$. In order to show that $\psi$ induces a monoid homomorphism $M(E,w)\rightarrow V(L_K(E,w))$ we have to check that $\psi$ preserves the relations (1), i.e.
\begin{align}
&\psi(q^v_{l-1}+(w_l(v)-w_{l-1}(v))v)=\psi(q_l^v+\sum\limits_{\substack{e\in s^{-1}(v),\\w(e)=w_l(v)}}r(e))\nonumber
\\
\Leftrightarrow &\left[\begin{pmatrix}\epsilon_{l-1}(v)&&&\\&v&&\\&&\ddots&\\&&&v\end{pmatrix}\right]=\left[\begin{pmatrix}\epsilon_{l}(v)&&&\\&r(e^{n_{l-1}(v)+1,v})&&\\&&\ddots&\\&&&r(e^{n_l(v),v})\end{pmatrix}\right]
\end{align}
for any $v\in E^0$ and $1\leq l \leq k_v$ (where $\epsilon_{0}(v)$ and $\epsilon_{k_v}(v)$ are the empty matrix). Set
\[X_l(v):=\begin{pmatrix}\epsilon_l(v)& A^{n_{l-1},n_l}_{w_0,w_l}(v)\end{pmatrix}\text{ and }Y_l(v):=(X_l(v))^*=\begin{pmatrix}\epsilon_l(v)\\ (A^{n_{l-1},n_l}_{w_0,w_l}(v))^*\end{pmatrix}.\]
Clearly 
\begin{align}
X_l(v)Y_l(v)&=\epsilon_l(v)+A^{n_{l-1},n_l}_{w_0,w_l}(v)(A^{n_{l-1},n_l}_{w_0,w_l}(v))^*.
\end{align}
Writing $A^{n_{0},n_l}_{w_0,w_l}(v)$ in block form
\[A^{n_{0},n_l}_{w_0,w_l}(v)=\begin{pmatrix}A^{n_{0},n_{l-1}}_{w_0,w_{l-1}}(v)&A^{n_{l-1},n_l}_{w_0,w_{l-1}}(v)\\0&A^{n_{l-1},n_{l}}_{w_{l-1},w_l}(v)\end{pmatrix}\]
it is easy to deduce from equation (4) that 
\begin{equation}
\epsilon_l(v)=\begin{pmatrix}\epsilon_{l-1}(v)&&&\\&v&&\\&&\ddots&\\&&&v\end{pmatrix}-A^{n_{l-1},n_l}_{w_0,w_{l}}(v)(A^{n_{l-1},n_l}_{w_0,w_{l}}(v))^*.
\end{equation}
By equations (7) and (8) we have
\begin{equation*}
X_l(v)Y_l(v)=\begin{pmatrix}\epsilon_{l-1}(v)&&&\\&v&&\\&&\ddots&\\&&&v\end{pmatrix}.
\end{equation*}
On the other hand one checks easily that 
\[Y_l(v)X_l(v)=\begin{pmatrix}\epsilon_{l}(v)&&&\\&r(e^{n_{l-1}(v)+1,v})&&\\&&\ddots&\\&&&r(e^{n_l(v),v})\end{pmatrix}\]
(note that $(A^{n_{l},n_{k_v}}_{w_0,w_l}(v))^*A^{n_{l-1},n_l}_{w_0,w_l}(v)=0$ by equation (3); hence $\epsilon_l(v)A^{n_{l-1},n_l}_{w_0,w_l}(v)=0$ and $(A^{n_{l-1},n_l}_{w_0,w_l}(v))^*\epsilon_l(v)=0$). Thus equation (6) holds for any $v\in E^0$ and $1\leq l \leq k_v$ and therefore $\psi$ induces a monoid homomorphism $\theta_{(E,w)}:M(E,w)\rightarrow V(L_K(E,w))$. It is an easy exercise to show that $\theta:M\rightarrow V\circ L_K$ is a natural transformation (note that for $v\in E^0$ and $1\leq l\leq k_v-1$, the matrix $\epsilon_l(v)$ does not depend on the weight-respecting order of the elements of $s^{-1}(v)$ chosen in the second line of Part I).\\
\\
{\bf Part II} We want to show that the natural transformation $\theta:M\rightarrow V\circ L_K$ defined in Part I is a natural isomorphism, i.e. that $\theta_{(E,w)}:M(E,w)\rightarrow V(L_K(E,w))$ is an isomorphism for any weighted graph $(E,w)$. By \cite[Lemma 5.19]{hazrat13} any weighted graph is a direct limit of a direct system of finite weighted graphs. Hence it is sufficient to show that $\theta_{(E,w)}$ is an isomorphism for any finite weighted graph $(E,w)$ (note that $M$, $V$ and $L_K$ commute with direct limits).\\
Let $(E,w)$ be a finite weighted graph. Set $B_0:=K^{E^0}$. We denote by $\alpha_v$ the element of $B_0$ whose $v$-component is $1$ and whose other components are $0$. Let $\{v_1, \dots, v_m\}$ be the elements of $E^0$ which emit vertices. Let $1\leq t \leq m$ and assume that $B_{t-1}$ has already been defined. We define an associative $K$-algebra $B_t$ as follows. Set $C_{t,0}:=B_{t-1}$ and let $\beta^{t,0}:C_{t,0}\rightarrow C_{t,0}$ be the map sending any element to $0$. For $1\leq l \leq k_{v_t}-1$ define inductively $C_{t,l}:=C_{t,l-1}\langle \beta^{t,l}: \overline{O_{t,l}}\rightarrow \overline{O_{t,l}};(\beta^{t,l})^2=\beta^{t,l}\rangle $ (see \cite[p. 39]{bergman74}) where 
\[O_{t,l}=\im (\beta^{t,l-1})\oplus \bigoplus\limits_{h=w_{l-1}(v_t)+1}^{w_l(v_t)}  \alpha_{v_t}C_{t,l-1}.\]
Set $D_{t,0}:=C_{t,k_{v_t}-1}$. For $1\leq l \leq k_{v_t}-1$ define inductively $D_{t,l}:=D_{t,l-1}\langle \gamma^{t,l},(\gamma^{t,l})^{-1}:\overline{P_{t,l}}\cong\overline{Q_{t,l}}\rangle $ (see \cite[p. 38]{bergman74}) where 
\[P_{t,l}=\bigoplus\limits_{h=n_{l-1}(v_t)+1}^{n_l(v_t)}\alpha_{r(e^{h,v_t})}D_{t,l-1}\text{ and }Q_{t,l}=\ker(\beta^{t,l}).\]
Finally define $B_t:=D_{t,k_{v_t}-1}\langle \gamma^{t,k_{v_t}},(\gamma^{t,k_{v_t}})^{-1}:\overline{P_{t,k_{v_t}}}\cong\overline{Q_{t,k_{v_t}}}\rangle $ where
\begin{align*}
&P_{t,l}=\bigoplus\limits_{h=n_{k_{v_t}-1}(v_t)+1}^{n_{k_{v_t}}(v_t)}\alpha_{r(e^{h,v_t})}D_{t,k_{v_t}-1}\text{ and }\\
&Q_{t,k_{v_t}}=\im(\beta^{t,k_{v_t}-1})\oplus\bigoplus\limits_{h=w_{k_{v_t}-1}(v_t)+1}^{w_{k_{v_t}}(v_t)}  \alpha_{v_t}D_{t,k_{v_t}-1}.
\end{align*}
We will show that $L_K(E,w)\cong B_m$.\\
Investigating the proofs of \cite[Theorems 3.1, 3.2]{bergman74} we see that $B_{m}$ is presented by the generating set 
\begin{align*}
X:=&\{\alpha_{v}\mid v\in E_0\}\cup\{\beta^{t,l}_{i,j}\mid 1\leq t \leq m, 1\leq l\leq k_{v_t}-1, 1\leq i,j\leq w_l(v_t)\}\\
&\cup \{\gamma^{t,l}_{i,j},(\gamma^{t,l}_{j,i})^*\mid 1\leq t \leq m, 1\leq l\leq k_{v_t}, 1\leq i\leq w_l(v_t), 1\leq j \leq n_l(v_t)-n_{l-1}(v_t)\}
\end{align*}
and the relations
\begin{center}
\begin{tabular}{r l l}
(i)&$\alpha_u\alpha_v=\delta_{uv}\alpha_u$&$(u,v\in E^0)$,\\
(ii)& $\id_{O_{t,l}}\beta^{t,l}=\beta^{t,l}=\beta^{t,l}\id_{O_{t,l}}$&$(1\leq t \leq m, 1\leq l\leq k_{v_t}-1),$\\
(iii)&$(\beta^{t,l})^2=\beta^{t,l}$&$(1\leq t \leq m, 1\leq l\leq k_{v_t}-1),$\\
(iv)&$\gamma^{t,l}\id_{P_{t,l}}=\gamma^{t,l}=\id_{Q_{t,l}}\gamma^{t,l}$&$(1\leq t \leq m, 1\leq l\leq k_{v_t})$,\\
(v)& $(\gamma^{t,l})^*\id_{Q_{t,l}}=(\gamma^{t,l})^*=\id_{P_{t,l}}(\gamma^{t,l})^*$&$(1\leq t \leq m, 1\leq l\leq k_{v_t}),$\\
(vi)&$\gamma^{t,l}(\gamma^{t,l})^*=\id_{Q_{t,l}}$&$(1\leq t \leq m, 1\leq l\leq k_{v_t}),$\\
(vii)&$(\gamma^{t,l})^*\gamma^{t,l}=\id_{P_{t,l}}$&$(1\leq t \leq m, 1\leq l\leq k_{v_t})$
\end{tabular}
\end{center}
where $\beta^{t,l}\in\Mat_{w_l(v_t)}(K\X)$ (we denote by $K\X$ the free associative $K$-algebra generated by $X$) is the matrix whose entry at position $(i,j)$ is $\beta^{t,l}_{i,j}$, $\gamma^{t,l}\in\Mat_{w_l(v_t)\times (n_l(v_t)-n_{l-1}(v_t))}(K\X)$ is the matrix whose entry at position $(i,j)$ is $\gamma^{t,l}_{i,j}$, $(\gamma^{t,l})^*\in\Mat_{(n_l(v_t)-n_{l-1}(v_t))\times w_l(v_t)}(K\X)$ is the matrix whose entry at position $(i,j)$ is $(\gamma^{t,l}_{j,i})^*$ and further
\begin{align*}
\id_{O_{t,l}}&=\begin{pmatrix}
\beta^{t,l-1}&&&\\
&\alpha_{v_t}&&\\
&&\ddots&\\
&&&\alpha_{v_t}
\end{pmatrix}\in\Mat_{w_l(v_t)}(K\X),\\
\id_{P_{t,l}}&=\begin{pmatrix}
\alpha_{r(e^{n_{l-1}(v_t)+1,v_t})}&&\\
&\ddots&\\
&&\alpha_{r(e^{n_{l}(v_t),v_t})}
\end{pmatrix}\in\Mat_{n_{l}(v_t)-n_{l-1}(v_t)}(K\X),\\
\id_{Q_{t,l}}&=\id_{O_{t,l}}-\beta^{t,l}\in\Mat_{w_l(v_t)}(K\X)\text{ if } l<k_{v_t}\text{ and }\\
\id_{Q_{t,k_{v_t}}}&=\begin{pmatrix}
\beta^{t,k_{v_t}-1}&&&\\
&\alpha_{v_t}&&\\
&&\ddots&\\
&&&\alpha_{v_t}
\end{pmatrix}\in\Mat_{w_{k_{v_t}}(v_t)}(K\X)
\end{align*}
(we let $\beta^{t,0}$ be the empty matrix). Define an $K$-algebra homomorphism $\zeta:L_K(E,w)\rightarrow B_m$ by \\
\[\zeta(v)=\alpha_v~(v\in E^0),\quad\zeta(A_{w_0,w_l}^{n_{l-1},n_l}(v_t))=\gamma^{t,l},\quad\zeta((A_{w_0,w_l}^{n_{l-1},n_l}(v_t))^*)=(\gamma^{t,l})^*~(1\leq t\leq m,1\leq l\leq k_{v_t})\]\\
(meaning that each entry of $A_{w_0,w_l}^{n_{l-1},n_l}(v_t)$ (resp. $(A_{w_0,w_l}^{n_{l-1},n_l}(v_t))^*$) is mapped to the corresponding entry of $\gamma^{t,l}$ (resp. $(\gamma^{t,l})^*$). Define an $K$-algebra homomorphism $\xi:B_m\rightarrow L_K(E,w)$ by\\ 
\begin{align*}
&\xi(\alpha_v)=v~(v\in E^0),\quad \xi(\beta^{t,l})=\epsilon_l(v_t)~(1\leq t\leq m,1\leq l\leq k_{v_t}-1)\\
 &\xi(\gamma^{t,l})=A_{w_0,w_l}^{n_{l-1},n_l}(v_t), \quad\xi((\gamma^{t,l})^*)=(A_{w_0,w_l}^{n_{l-1},n_l}(v_t))^*~(1\leq t\leq m,1\leq l\leq k_{v_t}).
\end{align*}\\
We leave it to the reader to show that $\zeta$ and $\xi$ are well-defined and further $\xi\circ\zeta=\id_{L_K(E,w)}$ and $\zeta\circ\xi=\id_{B_m}$ (a hint: in order to show that $\zeta(\xi(\beta^{t,l}))=\beta^{t,l}$, it is convenient to use equation (8) and relation (vi) above). Thus $L_K(E,w)\cong B_m$.\\
By \cite[Theorems 5.1, 5.2]{bergman74}, the abelian monoid $V'(B_m)$ (see Remark \ref{rempro}) is presented by the generating set $\{v,p_1^v,\dots,p^v_{k_v-1},q_1^v,\dots,q^v_{k_v-1}\mid v\in E^0\}$ and the relations
\begin{enumerate}[(i)]
\item $q^v_{i-1}+(w_i(v)-w_{i-1}(v))v=q_i^v+p_i^v\quad(v\in E^0,1\leq i\leq k_v-1)$,
\item $p_i^v=\sum\limits_{\substack{e\in s^{-1}(v),\\w(e)=w_i(v)}}r(e)\quad(v\in E^0,1\leq i\leq k_v-1)$ and
\item $q^v_{k_v-1}+(w_{k_v}(v)-w_{k_v-1}(v))v=\sum\limits_{\substack{e\in s^{-1}(v),\\w(e)=w_{k_v}(v)}}r(e)\quad(v\in E^0)$
\end{enumerate}
where $q^v_0=0$. It follows from Lemma \ref{lemmon} that $M(E,w)\cong V'(B_m)\cong V'(L_K(E,w))\cong V(L_K(E,w))$. One checks easily that the monoid isomorphism $M(E,w)\rightarrow V(L_K(E,w))$ one gets in this way is precisely $\theta_{(E,w)}$. \\
Furthermore, the right global dimension of $B_m\cong L_K(E,w)$ is $\leq 1$ by \cite[Theorems 5.1, 5.2]{bergman74}, i.e. $L_K(E,w)$ is right hereditary. Since $L_K(E,w)$ is a ring with involution, we have $L_K(E,w)\cong L_K(E,w)^{op}$. Thus $L_K(E,w)$ is also left hereditary.
\end{proof}
\begin{corollary}\label{cornum}
Let $(E,w)$ be a weighted graph. If there is a vertex $v\in E^0$ such that $k_v>1$ (i.e. there are $e,f\in s^{-1}(v)$ such that $w(e)\neq w(f)$), then $|V(L_K(E,w))|=\infty$.
\end{corollary}
\begin{proof}
Let $v\in E^0$ be a vertex such that $k_v>1$. For any $n\in \N_0$ let $[nq_1^v]$ denote the equivalence class of $nq^v_1$ in $M(E,w)$. One checks easily that $[nq_1^v]=\{nq_1^v\}$. Hence the elements $[nq_1^v]~(n\in\N_0)$ are pairwise distinct in $M(E,w)$ (and therefore we have an embedding $\N_0\hookrightarrow M(E,w)$ defined by $n\mapsto [nq_1^v]$). It follows from Theorem \ref{thmm} that $|V(L_K(E,w))|=\infty$.
\end{proof}

In \cite[Section 4]{hazrat-preusser} it was ``proved" by using the false \cite[Theorem 5.21]{hazrat13}, that if $(E,w)$ is an LV-rose (see \cite[Definition 38]{hazrat-preusser}) such that the minimal weight is $2$, the maximal weight is $l\geq 3$ and the number of edges is $l+m$ for some $m > 0$, then the domain $L_K(E,w)$ is not isomorphic to any of the Leavitt algebras $L_K(n,n+k)$ where $n,k\geq 1$. Using Theorem \ref{thmm} we prove a stronger statement:
\begin{corollary}
Let $(E,w)$ be an LV-rose such that there are edges of different weights. Then $L_K(E,w)$ is a domain that is neither $K$-algebra isomorphic to a Leavitt path algebra $L_K(F)$ nor to a Leavitt algebra $L_K(n,n+k)$.
\end{corollary}
\begin{proof}
First we show that $L_K(E,w)$ is not isomorphic to a Leavitt path algebra. By \cite[Theorem 41]{hazrat-preusser}, $L_K(E,w)$ is a domain (i.e. a nonzero ring without zero divisors). It is well-known that if $F$ is a directed graph such that $L_K(F)$ is a domain, then $F$ is either the graph $\xymatrix{\bullet}$ and we have $L_K(F)\cong K$, or the graph $\xymatrix{\bullet\ar@(ur,dr)}$\hspace{0.5cm} and we have $L_K(F)\cong K[x,x^{-1}]$. In both cases we have $V(L_K(F))\cong\N_0$ by Example \ref{exex1} and Theorem \ref{thmm}. Assume that there is an isomorphism $\phi:\N_0\rightarrow M(E,w)$. One checks easily that if $q_1^v=a+b$ for some $a,b\in M(E,w)$, then $a=0$ and $b=q_1^v$ or vice versa. Hence $\phi(1)=q_1^v$. But then $\phi$ cannot be surjective (see the proof of the previous corollary). Hence $V(L_K(E,w))\not\cong \N_0$ and therefore $L_K(E,w)$ is not isomorphic to a Leavitt path algebra $L_K(F)$.\\
Next we show that $L_K(E,w)$ is not isomorphic to a Leavitt algebra $L_K(n,n+k)$ where $n\geq 1$ and $k\geq 0$. It follows from Example \ref{wlpapp} and Theorem \ref{thmm} that $V(L_K(n,n+k))\cong \N_0/\langle n=n+k\rangle$. If $k=0$, then $V(L_K(n,n+k))\cong \N_0$ and therefore $L_K(E,w)$ is not isomorphic to $L_K(n,n+k)$ by the previous paragraph. Suppose now that $k\geq 1$. Then $|V(L_K(n,n+k))|=n+k<\infty$. But by Corollary \ref{cornum}, $|V(L_K(E,w))|=\infty$. Hence $L_K(E,w)$ is not isomorphic to a Leavitt algebra $L_K(n,n+k)$.
\end{proof}

Now we consider $K_0$ of a wLpa. Let $(E,w)$ denote a weighted graph. Since $L_K(E,w)$ is clearly a ring with local units, $K_0(L_K(E,w))$ is the group completion $(V(L_K(E,w)))^+$ of the abelian monoid $V(L_K(E,w))$, see \cite[p. 77]{abrams-ara-molina}. By Theorem \ref{thmm}, $(V(L_K(E,w)))^+\cong (M(E,w))^+$. It follows from \cite[Equation (45)]{hazrat13} that $(M(E,w))^+$ is presented as abelian group by the generating set $\{v,q_1^v,\dots,q^v_{k_v-1}\mid v\in E^0\}$ and the relations
\begin{equation*}
q^v_{i-1}+(w_i(v)-w_{i-1}(v))v=q_i^v+\sum\limits_{\substack{e\in s^{-1}(v),\\w(e)=w_i(v)}}r(e)\quad\quad(v\in E^0,1\leq i\leq k_v)
\end{equation*}
where $q^v_0=q^v_{k_v}=0$. We can rewrite the relations above in the form 
\begin{align*}
&q_i^v=q^v_{i-1}+(w_i(v)-w_{i-1}(v))v-\sum\limits_{\substack{e\in s^{-1}(v),\\w(e)=w_i(v)}}r(e)\quad\quad(v\in E^0,1\leq i\leq k_v).
\end{align*}
By successively applying Lemma \ref{lemmon} we get that $(M(E,w))^+$ is presented by the generating set $E^0$ and the relations 
\[w(v)v=\sum\limits_{e\in s^{-1}(v)}r(e)\quad\quad(v\in E^0).\] 
It follows that \cite[Theorem 5.23]{hazrat13} is correct!
\section{Examples}
\begin{example}
Consider the weighted graph\\
\[
(E,w):\xymatrix@C+15pt{ u& v\ar[l]_{e,1}\ar[r]^{f,2}& x}.
\]\\
As mentioned at the beginning of the previous section, $L_K(E,w)\cong \Mat_3(K)\oplus \Mat_3(K)$.
By Theorem \ref{thmm} and Lemma \ref{lemmon}, 
\[V(L_K(E,w))\cong \N_0^{\{u,v,q_1^v,x\}}/\langle \alpha_v=\alpha_{q_1^v}+\alpha_u,\alpha_{q_1^v}+\alpha_v=\alpha_x\rangle\cong \N_0^2.\]
\end{example}
\begin{example}
Consider the weighted graph
\[
(E,w):\xymatrix@C+15pt{ u& v\ar@/_1.7pc/[l]_{e,1}\ar@/^1.7pc/[l]^{f,2}}.
\]
Let $F$ be the directed graph
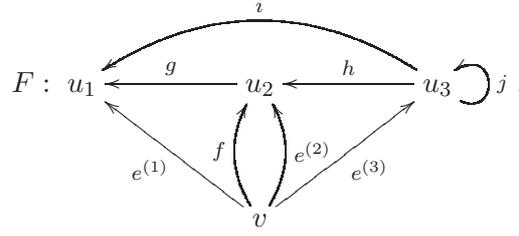
\[
F:\xymatrix@R+15pt@C+25pt{ u_1&u_2\ar[l]_{g}&u_3\ar@(dr,ur)_{j}\ar[l]_{h}\ar@/_2pc/[ll]_{i}\\ &v\ar@/_0.8pc/[u]_{e^{(2)}}\ar@/^0.8pc/[u]^{f}\ar[ul]^{e^{(1)}}\ar[ur]_{e^{(3)}}&}.
\]
There is a $*$-algebra isomorphism $L_K(E,w)\rightarrow L_K(F)$ mapping $u\mapsto \sum u_i$, $v\mapsto v$, $e_1 \mapsto \sum e^{(i)}$, $f_1\mapsto f$ and $f_2\mapsto fg+e^{(1)}i^*+e^{(2)}h^* +e^{(3)}j^*$.  
By Theorem \ref{thmm} and Lemma \ref{lemmon}, 
\[V(L_K(E,w))\cong \N_0^{\{u,v,q_1^v\}}/\langle \alpha_v=\alpha_{q_1^v}+\alpha_u,\alpha_{q_1^v}+\alpha_v=\alpha_u\rangle\cong \N_0^2/\langle(1,0)=(1,2)\rangle.\]
\cite[Theorem 3.2.5]{abrams-ara-molina} (or Theorem \ref{thmm}, which generalises \cite[Theorem 3.2.5]{abrams-ara-molina}) yields the same result for $V(L_K(F))$. 
\end{example}
\begin{example}
Consider the weighted graph\\
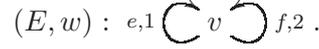
\[
(E,w):\xymatrix@C+15pt{ v\ar@(dl,ul)^{e,1}\ar@(dr,ur)_{f,2}}.
\]\\
By Theorem \ref{thmm}, 
\[V(L_K(E,w))\cong \N_0^{\{v,q_1^v\}}/\langle \alpha_v=\alpha_{q_1^v}+\alpha_v,\alpha_{q_1^v}+\alpha_v=\alpha_v\rangle\cong \N_0^2/\langle(1,0)=(1,1)\rangle.\]
Let $F$ be the directed graph\\
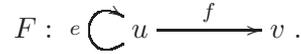
\[
F:\xymatrix@C+15pt{u\ar@(dl,ul)^{e}\ar[r]^{f}&v}.
\]\\
Its Leavitt path algebra $L_K(F)$ is called {\it algebraic Toeplitz $K$-algebra}, see \cite[Example 1.3.6]{abrams-ara-molina}. By \cite[Theorem 3.2.5]{abrams-ara-molina} we have $V(L_K(F))\cong V(L_K(E,w))$. But $\GKdim L_K(F)=2$ by \cite[Theorem 5]{zel12} while $\GKdim L_K(E,w)=\infty$ by \cite[Theorem 22]{preusser}. Hence $L_K(F)\not\cong L_K(E,w)$.
\end{example}
\begin{example}
Consider the LV-roses
\[(E,w):\xymatrix{
v \ar@(ur,dr)^{e,3} \ar@(dr,dl)^{f,3} \ar@(dl,ul)^{g,3}\ar@(ul,ur)^{h,3}& 
}\quad \text{ and }\quad (E,w'):\xymatrix{
 v \ar@(ur,dr)^{e,2} \ar@(dr,dl)^{f,3} \ar@(dl,ul)^{g,3}\ar@(ul,ur)^{h,3}& 
}.
\]
By Theorem \ref{thmm}, 
\[V(L_K(E,w))\cong \N_0^{\{v\}}/\langle 3\alpha_v=4\alpha_v\rangle\]
and
\[V(L_K(E,w'))\cong \N_0^{\{v,q_1^v\}}/\langle 2\alpha_v=\alpha_{q_1^v}+\alpha_v,\alpha_{q_1^v}+\alpha_v=3\alpha_v\rangle.\]\\
Since the image of $L_K(E,w)$ in $V(L_K(E,w))$ (resp. of $L_K(E,w')$ in $V(L_K(E,w'))$) is $\alpha_v$ (see Remark \ref{rempro}), the module type of $L_K(E,w)$ is $(3,1)$ and the module type of $L_K(E,w')$ is $(2,1)$.
\end{example}


\begin{thebibliography}{99}


\bibitem{abrams-ara-molina} G. Abrams, P. Ara, M. Siles Molina, Leavitt path algebras, Lecture Notes in Mathematics {\bf 2191}, Springer, 2017.

\bibitem{zel12}  A. Alahmadi, H. Alsulami, S. Jain, E. Zelmanov, \emph{Leavitt path algebras of finite Gelfand-Kirillov dimension},  J. Algebra Appl. {\bf 11} (2012), no. 6, 1250225.





\bibitem{bergman74} G. M. Bergman, \emph{Coproducts and some universal ring constructions}, 
Trans. Amer. Math. Soc. {\bf 200} (1974), 33--88.







\bibitem{hazrat13} R. Hazrat, \emph{The graded structure of Leavitt path algebras}, Israel J. Math. {\bf 195} (2013), no. 2, 833--895. 


\bibitem{hazrat-preusser} R. Hazrat, R. Preusser, \emph{Applications of normal forms for weighted Leavitt path algebras: simple rings and domains}, Algebr. Represent. Theor. {\bf 20} (2017), 1061–-1083. 











\bibitem{preusser} R. Preusser, \emph{Weighted Leavitt path algebras of finite Gelfand-Kirillov dimension}, arXiv:1804.09287 [math.RA]. 

\bibitem{preusser1} R. Preusser, \emph{Locally finite weighted Leavitt path algebras}, 
arXiv:1806.06139 [math.RA]. 




\end{thebibliography}
\end{document}